\documentclass[a4paper,12pt]{article}

\usepackage[table,dvipsnames,svgnames,x11names,hyperref]{xcolor}
\usepackage[utf8]{inputenc}

\usepackage{lmodern}
\usepackage[T1]{fontenc}

\usepackage[top=1in, bottom=1.25in, left=1.25in, right=1.25in]{geometry}
\usepackage{amsmath,amssymb,amsthm}

\usepackage{tikz}
\usepackage{array}
\usepackage{enumerate}
\usepackage{graphicx}
\usepackage{float}
\usepackage{caption}
\usepackage{subcaption}
\usepackage[colorlinks=true, allcolors=MidnightBlue]{hyperref}
\usepackage{rotating}

\newtheorem{thm}{Theorem}
\newtheorem{cor}[thm]{Corollary}
\newtheorem{lem}[thm]{Lemma}
\newtheorem{prop}[thm]{Proposition}

\theoremstyle{definition}

\theoremstyle{remark}
\newtheorem{rem}{Remark}
\newtheorem{example}{Example}
\newtheorem{question}{Question}

\DeclareMathOperator{\sd}{sd}

\author{Sel\c{c}uk Kayacan
  \thanks{The author thanks to Volkmar Welker who made many useful suggestions during the preparation of this manuscript. This study was supported by Scientific and Technological Research Council of Turkey (TUBITAK) under the Grant Number 122F490. The author thanks to TUBITAK for their supports.}
}
\title{Subrack Lattices of Conjugation Racks}
\date{}

\begin{document}

\maketitle

\small

\begin{center}
  Bahçeşehir University, Faculty of Engineering\\ and Natural Sciences,
  Istanbul, Turkey\\
  {\it e-mail:} \href{mailto:selcuk.kayacan@bau.edu.tr}{selcuk.kayacan@bau.edu.tr}
\end{center}

\begin{abstract}
  A rack is a set with a binary operation such that left multiplications are automorphisms of the set and a quandle is a rack satisfying a certain condition. Let $S$ be a subset of a finite group $G$ which is closed under the conjugation operation $a\triangleright b := aba^{-1}$. The set $S$ with the conjugation operation $\triangleright$ is a quandle. We call those objects \emph{conjugation racks}. The prime examples are
  \begin{itemize}
  \item the group rack $(G,\triangleright)$,
  \item the conjugacy class rack $(C,\triangleright)$, where $C$ is a conjugacy class in $G$, and
  \item the $p$-power rack $(G_p,\triangleright)$, where $p$ is a prime and $G_p$ is the set of all elements in $G$ whose order is a power of $p$.
  \end{itemize}
  The set of all subracks of a finite rack form a lattice under inclusion. In this paper we study the subrack lattices of the conjugation racks. In particular, we show that the subrack lattice can be associated with a subposet of a partition lattice as well as with a subposet of an integer partition lattice in a canonical way if the rack is connected. And, if the rack is not connected, the study of the homotopy properties of the subrack lattice can be reduced into the study of the homotopy properties of the subposet of parabolic subracks. We also prove that for a certain class of $p$-power racks the order of a Sylow $p$-subgroup divides the reduced Euler characteristic of the subrack lattice of the $p$-power rack. This statement can be considered as the rack analogue of a result by Brown in the field of subgroup complexes regarding the Euler characteristic of the poset of nontrivial $p$-subgroups of a group.
  


\end{abstract}

\section{Introduction}

There are many works in literature studying the combinatorial objects that are associated with the algebraic objects in a natural way. For example, if the algebraic object is a group $G$, one naturally defines the \emph{subgroup lattice} $\mathcal{L}(G)$ as the lattice of all subgroups of $G$. In this context we may ask the following question: Given the subgroup lattice, how much information can be recovered about the group itself? An early result in this direction by Baer is that finite abelian groups can almost be distinguished by their subgroup lattices (see \cite{Baer39} or \cite[Corollary~1.2.8]{Sch94}).

A relevant combinatorial object is the \emph{intersection graph} $\Gamma(G)$ of a group $G$: The vertex set of $\Gamma(G)$ is the set of proper nontrivial subgroups of $G$ and there is an edge between two vertices if and only if they are distinct and their intersection is nontrivial. Observe that given $\Gamma(G)$ one can recover $\mathcal{L}(G)$, but not vice versa. Intuitively, intersection graphs should be highly connected graphs and if there are some examples of such graphs with `low' connectivity, they must be exceptional. The main results in \cite{Kay18a} made this intuition more precise. It was proven in \cite[Theorem~5.2]{KY15} that finite abelian groups can be distinguished by their intersection graphs in the same way as their subgroup lattices. Therefore, we may conclude that subgroup lattices and intersection graphs hold the same amount of information on the subgroup structure when the groups in question are abelian.

Given a partially ordered set (abbreviated poset) $\mathcal{P}$, the \emph{order complex} $\Delta(\mathcal{P})$ is defined as the simplicial complex whose faces are the totally ordered sets (chains) in $\mathcal{P}$. Introducing the order complexes is a standard way to further incorporate topological arguments into our combinatorial setting (see, for example, \cite{Br75,Qui78,HI89,Smi11}). Let $G$ be a group and $p$ be a prime dividing the order of $G$. We define $\mathcal{S}_p(G)$ as the poset of nontrivial $p$-subgroups of $G$. An important result by Brown in the field of subgroup complexes (see \cite{Br75} or \cite[Theorem~5.3.1]{Smi11}) states that the Euler characteristic $\chi(\Delta(\mathcal{S}_p(G)))\equiv 1$ modulo $p^k$, where $p^k$ is the $p$-part of $|G|$. Similar work using subgroup lattices, frames, coset posets, and quandles has also appeared in literature (see, \cite{SW12,SW16,Fum09,HSW19}).

In \cite{Kay18b} the \emph{intersection complex} $K(G)$ of a group $G$ is defined as the simplicial complex whose simplices are the sets of proper subgroups of $G$ which intersect nontrivially. The intersection graph $\Gamma(G)$ is the $1$-skeleton of the intersection complex $K(G)$. As is shown in \cite{Kay18b} $K(G)$ is contractible if the domination number of $\Gamma(G)$ is $1$. Notice that $K(G)$ is a crosscut complex of $\mathcal{L}(G)$ and therefore shares the homotopy type with $\Delta(\overline{\mathcal{L}(G)})$, where $\overline{\mathcal{L}(G)}:=\mathcal{L}(G)\setminus \{G,1\}$. 

A \emph{rack} $X$ is a set (possibly empty) together with a binary operation $\triangleright$ satisfying the following two axioms:
\begin{itemize}
  \setlength{\itemindent}{1em}
\item[(A1)] for all $a,b,c\in X$ we have $a \triangleright (b \triangleright c) = (a \triangleright b) \triangleright (a \triangleright c)$ 
\item[(A2)] for all $a,b\in X$ there is a
  unique $x\in X$ such that $a \triangleright x = b$.
\end{itemize}
By (A1) the map $\phi_a\colon X\to X;\,x\mapsto a\triangleright x$ is a rack morphism for every $a\in X$ and by (A2) those morphisms are bijections. Therefore, we may define a rack as a set with a binary operation such that left multiplications are automorphisms of the set. We say $X$ is a \emph{quandle} if, in addition, the following third axiom holds:
\begin{itemize}
  \setlength{\itemindent}{1em}
\item[(A3)] for every $a\in X$ we have $a\triangleright a = a$.
\end{itemize}
The defining axioms of racks and quandles are intimately related with the Reidemeister moves and those algebraic objects proved to be useful in the study of knot invariants (see, for example, \cite{Joy82,Mat82,FR92}).

Let $X$ be a finite rack. The set of subracks of $X$, denoted $\mathcal{R}(X)$, form a lattice under inclusion (see \cite{HSW19}). This lattice shown to be atomic (see \cite{KS19}) and complemented (see \cite{KS21}). We call $\mathcal{R}(X)$ the \emph{subrack lattice} of $X$. Let $G$ be a finite group and consider the conjugation operation $a \triangleright b := aba^{-1}$ defined on $G$. One can easily observe that the set $G$ together with the conjugation operation $\triangleright$ form a quandle. Given the subrack lattice $\mathcal{R}(G)$, how much information can be recovered about the group $G$? Heckenberger, Shareshian, and Welker proved that (see \cite[Theorem~1.1]{HSW19}) the isomorphism type of $\mathcal{R}(G)$ determines whether $G$ is abelian, nilpotent, supersolvable, solvable and simple. Further, $\mathcal{R}(G)$ determines the nilpotence class of $G$ (see \cite{Kay21}). Let $\overline{\mathcal{R}(G)} := \mathcal{R}(G)\setminus \{G,\emptyset\}$. Another result by the same authors is that (see \cite[Proposition~1.3]{HSW19}) the order complex $\Delta(\overline{\mathcal{R}(G)})$ of $\mathcal{R}(G)$ is homotopy equivalent to a $(c-2)$-sphere $S^{c-2}$, where $c$ is the number of conjugacy classes in $G$. 

Let $G$ be a finite group and $C$ be a conjugacy class of $G$. Clearly, the conjugacy class $C$ equipped with the conjugation operation $a\triangleright b := aba^{-1}$ is also a quandle. The structure of the subracks of a conjugacy class plays an important role in the classification of finite-dimensional pointed Hopf algebras (see \cite{AG03}). In the previous paragraph we stated that $\Delta(\overline{\mathcal{R}(G)})$ is homotopy equivalent to a sphere. A similar result is the following (see \cite[Corollary~8]{Kay22}): If $C$ is a conjugacy class of a noncentral element in a finite nilpotent group $G$, then $\Delta(\overline{\mathcal{R}(C)})$ is a $(m-2)$-sphere, where $m$ is the number of maximal elements of $\mathcal{R}(C)$. Actually both of those results can be derived as corollaries of a more general result \cite[Theorem~1]{Kay22}.

Although results in \cite{HSW19} about the subrack lattice $\mathcal{R}(X)$ are quite comprehensive for $X=G$, in case $X=C$ is a single conjugacy class, there are only scattered examples that exhibit the complexity of the question: The \emph{partition lattice} $\Pi_n$ is the set of all partitions of $[n]:=\{1,2,\dots,n\}$ ordered by refinement and the \emph{$k$-equal partition lattice} $\Pi_{n,k}$ is the lattice of all partitions $B_1| \dots | B_r$ of $[n]$ such that, for all $1\leq i\leq r$, either $|B_i| = 1$ or $|B_i|\geq k$. In \cite[Example~2.3]{HSW19} the subrack lattice $\mathcal{R}(T)$ is shown to be isomorphic to partition lattice $\Pi_n$, where $T$ is the conjugacy class of transpositions of the symmetric group $S_n$. Let $P$ be the rack of all $p$-cycles in the alternating group $A_n$. In \cite[Proposition~2.11]{HSW19} it is proved that (the order complexes of) $\mathcal{R}(P)$ and $\Pi_{n,p}$ are homotopy equivalent if $p < n-2$ is an odd prime number. Those results are extended in \cite[Theorem~2]{Kay22}: The subposet $\Pi_{n,k}^m$ of $\Pi_{n,k}$ is defined in the following way. An element of $\Pi_{n,k}$ is an element of $\Pi_{n,k}^m$ if and only if either it is the partition consisting of singletons or the number of elements that are lying in nonsingleton blocks is at least $mk$. Let $C$ be the conjugacy class of elements in $S_n$ whose cycle type is $(p^s,1^t)$. Then $\mathcal{R}(C)$ and $\Pi_{n,p}^s$ are homotopy equivalent when the integers $n,p$ and $s$ are subject to certain restrictions.

Let $G$ be a finite group and $S$ be a subset of $G$ which is closed under the conjugation operation $\triangleright$. More generally, the set $S$ together with $\triangleright$ is an example of a quandle. We say $(S,\triangleright)$ is a \emph{conjugation rack}. The following three classes of conjugation racks deserve particular attention:
  \begin{itemize}
  \item the \emph{group rack} $(G,\triangleright)$,
  \item the \emph{conjugacy class rack} $(C,\triangleright)$ where $C$ is a conjugacy class in $G$, and
  \item the \emph{$p$-power rack} $(G_p,\triangleright)$, where $p$ is a prime and $G_p$ is the set of all elements in $G$ whose order is a power of $p$.
  \end{itemize}
  In this paper we study the subrack lattices of the conjugation racks by focusing on the conjugacy class racks and the $p$-power racks. In Section~\ref{sec:pre}, we shall review some relevent notions and facts from the literature and prove some auxiliary results. In particular, we define $\mathsf{Inf}(\mathcal{L})$ in this section as the subposet of the lattice $\mathcal{L}$ whose elements are typically the meet of a set of coatoms of $\mathcal{L}$ and prove that the complexes $\Delta(\overline{\mathcal{L}})$ and $\Delta(\mathsf{Inf}(\mathcal{L}))$ share the same homotopy type (see Proposition~\ref{prop:inf}). In Section~\ref{sec:ccr} we shall focus on the conjugacy class racks in our selection of examples and discussion. However, our main results in this section (see Theorem~\ref{thm:pi} and Theorem~\ref{thm:ipi}) which relate $\mathsf{Inf}(\mathcal{R}(X))$ with a subposet of the partition lattice and a subposet of the integer partition lattice in a canonical way are valid for all finite connected racks. In Section~\ref{sec:ppr} we will study the $p$-power racks. To this end we will introduce the subposet $\mathcal{S}(X)$ of spherical subracks and the subposet $\mathcal{P}(X)$ of parabolic subracks. By its definition $\mathcal{S}(X)$ would be the empty poset when $X$ is a connected rack. The usefulness of those definitions will be justified by the fact that $\Delta(\overline{\mathcal{R}(X)})$ is homotopy equivalent to the join of the complexes $\Delta(\mathcal{S}(X))$ and $\Delta(\mathcal{P}(X))$ (see Proposition~\ref{prop:ord} and Theorem~\ref{thm:p}). Furthermore, the complex $\Delta(\mathcal{S}(X))$ is homeomorphic to a sphere; hence, the study of the homotopy properties of $\mathcal{R}(X)$ can be reduced into the study of the homotopy properties of $\mathcal{P}(X)$. As was mentioned before the Euler characteristic of the order complex $\Delta(\mathcal{S}_p(G))$ of the poset of nontrivial $p$-subgroups of a group $G$ is congruent to $1$ modulo $p^k$, where $p^k$ is the order of a Sylow $p$-subgroup. We prove an analogous statement holds for the $p$-power rack $G_p$ when the poset $\mathcal{P}(G_p)$ satisfies a certain condition (see Theorem~\ref{thm:euler}).

\section{Preliminaries}\label{sec:pre}

In this paper we will be concerned with finite racks. As was remarked earlier, for a rack $X$, the map $\phi_a\colon X\to X;\,x\mapsto a\triangleright x$ is an automorphism of $X$ for all $a\in X$. Moreover, the map $\Phi\colon a\mapsto \phi_a$ is a rack morphism between $X$ and $\Phi(X)$, where $\phi_a\triangleright \phi_b := \phi_a\phi_b\phi_a^{-1}$. We say $X$ is \emph{faithful} if the map $\Phi$ is injective. Notice that $X$ is isomorphic to a conjugation rack if it is faithful.

The \emph{inner automorphism group} of $X$, denoted $\mathsf{Inn}(X)$, is defined as the (normal) subgroup of the full automorphism group of $X$ generated by the elements of $\Phi(X)$, i.e., $\mathsf{Inn}(X) := \langle\Phi(X)\rangle$. The group $\mathsf{Inn}(X)$ acts on $X$ by automorphisms and $X$ can be written as the disjoint union of $\mathsf{Inn}(X)$-orbits. We say $X$ is \emph{connected} if the action of $\mathsf{Inn}(X)$ on $X$ is transitive, i.e., if $X$ consists of a single $\mathsf{Inn}(X)$-orbit.  The group $\mathsf{Inn}(X)$ also acts on $\Phi(X)$ by conjugation and the map $\Phi\colon a\mapsto \phi_a$ is a $\mathsf{Inn}(X)$-equivariant map between $X$ and $\Phi(X)$. Moreover,  $X$ is  faithful and connected if and only if it is isomorphic, as a rack, to a conjugacy class $C$ of a group $G$ so that $G = \langle C\rangle$ and $Z(G) = 1$. For further details see \cite[Proposition~3.2]{AG03} or \cite[Proposition~2]{Kay20}.

For a partially ordered set (poset) $\mathcal{P}$ and an element $x\in \mathcal{P}$, we define $\mathcal{P}_{\geq x}$ as the subposet of $\mathcal{P}$ whose elements are the elements of $\mathcal{P}$ that are greater than or equal $x$, i.e., $$ \mathcal{P}_{\geq x} := \{y\in \mathcal{P}\colon y\geq x\}. $$  In an analogous way, we define the subposets $\mathcal{P}_{> x}$, $\mathcal{P}_{< x}$, $\mathcal{P}_{\leq x}$. If $x\leq y$, we define the  \emph{closed interval} $[x,y]$ as the set of all elements $z\in\mathcal{P}$ satisfying $x\leq z \leq y$. Similarly, for $x,y\in\mathcal{P}$, we define the \emph{open interval} $(x,y)$ as the set $\{z\mid x<z<y\}$. Notice that $[x,y]=\mathcal{P}_{\geq x}\cap \mathcal{P}_{\leq y}$ if $x\leq y$ and $(x,y)=\mathcal{P}_{> x}\cap \mathcal{P}_{< y}$. We say $y$ \emph{covers} $x$ if $x<y$ and $(x,y)=\emptyset$. Suppose $\mathcal{P}$ is finite. The \emph{length} $\ell(\mathcal{C})$ of a chain (or a linearly ordered set) $\mathcal{C}$ in $\mathcal{P}$ is the number of elements in $\mathcal{C}$ minus $1$. In such a case we say $\mathcal{C}$ is a $\ell(\mathcal{C})$-chain. And the length $\ell(\mathcal{P})$ of $\mathcal{P}$ is the maximum of the lengths of chains in $\mathcal{P}$. The dual poset of $\mathcal{P}$, i.e., the poset $\mathcal{P}$ with the order relation reversed, is denoted by $\mathcal{P}^*$. Notice that the set of chains of $\mathcal{P}$ coincides with the set of chains of $\mathcal{P}^*$, so their order complexes are same. 

A poset $\mathcal{P}$ is called \emph{bounded} if it has a unique maximal element, say $\hat{1}$, and a unique minimal element, say $\hat{0}$. If $\mathcal{P}$ is a bounded poset, we define the proper part of $\mathcal{P}$ as $\overline{\mathcal{P}} := \mathcal{P}\setminus \{\hat{1},\hat{0}\}$. In particular, the proper part of the subrack lattice of a rack $X$ is $\overline{\mathcal{R}(X)}:=\mathcal{R}(X)\setminus \{X,\emptyset\}$. When we say the order complex of a bounded poset $\mathcal{P}$, we mean the order complex $\Delta(\overline{\mathcal{P}})$. Notice that if $\mathcal{P}$ is bounded, then $\Delta(\mathcal{P})$ would be a cone which is a contractible space. Given a poset $\mathcal{P}$ (which may be bounded or not) we define the bounded poset $\widehat{\mathcal{P}}:=\mathcal{P}\sqcup\{\hat{1},\hat{0}\}$ by adding a unique maximum element $\hat{1}$ and a unique minimum element $\hat{0}$. Notice that  $\mathcal{P}\sqcup\{\hat{1},\hat{0}\}$ is a disjoint union. 

\begin{lem}[see {\cite[Corollary~10.12]{Bjo95}}]\label{lem:poset}
Let $f\colon\mathcal{P}\to\mathcal{P}$ be a poset endomorphism such that $x \leq f(x)$ for all $x\in \mathcal{P}$. Then $f$ induces homotopy equivalence between $\Delta(\mathcal{P})$ and $\Delta(f(\mathcal{P}))$. 
\end{lem}

\begin{rem}
  Since the chains of a poset and its dual are same, by Lemma~\ref{lem:poset} we conclude that $\Delta(\mathcal{P})$ and $\Delta(f(\mathcal{P}))$ are homotopy equivalent if $f\colon\mathcal{P}\to\mathcal{P}$ is a poset endomorphism and $x \geq f(x)$ for all $x\in \mathcal{P}$.
\end{rem}

A \emph{closure operator} $f$ on a poset $\mathcal{P}$ is an endomorphism of $\mathcal{P}$ satisfying the following requirements:
\begin{enumerate}[(i)]
\item $x \leq f(x)$ for all $x\in \mathcal{P}$ (or, $x \geq f(x)$ for all $x\in \mathcal{P}$)
\item $f^2(x) = f(x)$ for all $x\in\mathcal{P}$
\end{enumerate}

\begin{rem}
  If $f\colon\mathcal{R}(X)\to\mathcal{R}(X)$ is a closure operator such $f(\overline{\mathcal{R}(X)})\subseteq \overline{\mathcal{R}(X)}$, then $f$ induces a homotopy equivalence between $\Delta(\overline{\mathcal{R}(X)})$ and $\Delta(f(\overline{\mathcal{R}(X)}))$.
\end{rem}

In \cite{Qui78} Quillen introduced several poset fiber theorems. The most basic of those theorems states that if $f\colon\mathcal{P}\to\mathcal{Q}$ is a poset map such that for all $y\in \mathcal{Q}$ the fiber $f^{-1}(\mathcal{Q}_{\leq y})$ is contractible, then $f$ induces a homotopy equivalence of $\mathcal{P}$ with $\mathcal{Q}$. A powerful generalization of this result is the following. For two topological spaces $T_1$ and $T_2$, we write $T_1\simeq T_2$ if $T_1$ and $T_2$ are homotopy equivalent.

\begin{thm}[see {\cite[Theorem~1.1]{BWW05}}]\label{thm:fiber}
  Let $f\colon \mathcal{P}\to \mathcal{Q}$ be a poset map such that for all $y\in \mathcal{Q}$ the fiber $\Delta(f^{-1}(\mathcal{Q}_{\leq y}))$ is $\ell(f^{-1}(\mathcal{Q}_{<y}))$-connected. Then $$ \Delta(\mathcal{P}) \simeq \Delta(\mathcal{Q})\vee \{\Delta(f^{-1}(\mathcal{Q}_{\leq y})) * \Delta(\mathcal{Q}_{>y})\mid y\in \mathcal{Q}\}, $$ where $\vee$ denotes the wedge (of $\Delta(\mathcal{Q})$ and all $\Delta(f^{-1}(\mathcal{Q}_{\leq y})) * \Delta(\mathcal{Q}_{>y})$) formed by identifying the vertex $y$ in $\Delta(\mathcal{Q})$ with any element of $f^{-1}(\mathcal{Q}_{\leq y})$, for each $y\in \mathcal{Q}$. Consequently, if $\Delta(\mathcal{Q})$ is connected, then $$ \Delta(\mathcal{P}) \simeq \Delta(\mathcal{Q}) \vee \bigvee_{y\in \mathcal{Q}} \Delta(f^{-1}(\mathcal{Q}_{\leq y})) * \Delta(\mathcal{Q}_{>y}).  $$
\end{thm}

The \emph{face poset} $\mathcal{P}(\Omega)$ of a simplicial complex $\Omega$ is the poset of all simplices of $\Omega$ ordered by containment $\subseteq$. The \emph{(first) barycentric subdivision} $\sd(\Omega)$ of a simplicial complex $\Omega$ is defined as $$\sd(\Omega):=\Delta(\mathcal{P}(\Omega)).$$
It is a standard fact that $\Omega$ and $\sd(\Omega)$ share the same homotopy type.

A subset $\mathcal{C}$ of a finite poset $\mathcal{P}$ is called a \emph{crosscut} if it has the following properties:
\begin{enumerate}[(i)]
\item $\mathcal{C}$ is an antichain, i.e., no two elements of $\mathcal{C}$ are comparable.
\item For every chain $\sigma$ in $\mathcal{P}$ there exists some element in $\mathcal{C}$ which is comparable to each element in $\sigma$.
\item If $\mathcal{A}\subseteq \mathcal{C}$ has an upper bound or a lower bound in $\mathcal{P}$, then the join $\bigvee A$ or the meet $\bigwedge A$ exists in $\mathcal{P}$.
\end{enumerate}
For example, the set of maximal subracks of a finite rack $X$ form a crosscut in  $\overline{\mathcal{R}(X)}$. Let $\mathcal{C}$ be a crosscut in $\mathcal{P}$. The \emph{crosscut complex} $\Delta(\mathcal{P},\mathcal{C})$ is defined as the simplicial complex whose simplices are the subsets of $\mathcal{C}$ having an upper bound or a lower bound in $\mathcal{P}$.

\begin{thm}[see {\cite[Theorem~3]{Rot64}} or {\cite[Theorem~10.8]{Bjo95}}]\label{thm:crosscut}
  The crosscut complex $\Delta(\mathcal{P},\mathcal{C})$ and $\Delta(\mathcal{P})$ are homotopy equivalent.
\end{thm}

Let $\mathcal{L}$ be a finite lattice. We define $\mathsf{Inf}(\mathcal{L})$ as the subposet of $\overline{\mathcal{L}}$ consisting of those elements that are meet of a set of coatoms of $\mathcal{L}$. (In \cite{HSW19} this poset is denoted by $\mathsf{Int}(\mathcal{L})$.) More generally, for a subset $\mathcal{S}$ of $\overline{\mathcal{L}}$, the poset $\mathsf{Inf}(\mathcal{L},\mathcal{S})$ is the subposet of $\overline{\mathcal{L}}$ whose elements are the ones that are the meet of a set of elements in $\mathcal{S}$.

\begin{prop}\label{prop:inf}
  For a finite lattice $\mathcal{L}$, its order complex $\Delta(\overline{\mathcal{L}})$ and the complex $\Delta(\mathsf{Inf}(\mathcal{L}))$ share the same homotopy type.
\end{prop}

\begin{proof}
  Let $\mathcal{C}$ be the set of coatoms of $\mathcal{L}$. Consider the map $\epsilon\colon \overline{\mathcal{L}}\to \overline{\mathcal{L}}$ defined in the following way: If $x\in \overline{\mathcal{L}}$ and $\mathcal{C}_x:=\{c\in \mathcal{C} \mid x\leq c\}$, then $\epsilon(x)$ is the meet of all elements in $\mathcal{C}_x$. It is easy to observe that $\epsilon$ is a closure operator. Since $\epsilon(\overline{\mathcal{L}}) = \mathsf{Inf}(\mathcal{L})$, by Lemma~\ref{lem:poset} the result follows.
\end{proof}

\begin{rem}
   Let $\mathcal{C}=\{c_1,\dots,c_m\}$ be the set of coatoms of $\mathcal{L}$. From the previous results we know that $\Delta(\overline{\mathcal{L}}) \simeq \Delta(\overline{\mathcal{L}},\mathcal{C})$ and $\Delta(\overline{\mathcal{L}},\mathcal{C}) \simeq \sd(\Delta(\overline{\mathcal{L}},\mathcal{C}))$. Let $\mathcal{F}$ be the face poset of $\Delta(\overline{\mathcal{L}},\mathcal{C})$, i.e., $\mathcal{F}$ is the poset $\mathcal{P}(\Delta(\overline{\mathcal{L}},\mathcal{C}))$. The map $\epsilon\colon \{c_{i_1},\dots,c_{i_k}\}\mapsto \{c_{j_1},\dots,c_{j_l}\}$, where $\{c_{j_1},\dots,c_{j_l}\}$ is the largest subset of $\mathcal{C}$ satisfying $c_{j_1}\wedge\dots\wedge c_{j_l} = c_{i_1}\wedge\dots\wedge c_{i_k}$, is a closure operator on $\mathcal{F}$. Since the order complex of $\mathcal{F}^*$ is the barycentric subdivision of $\Delta(\overline{\mathcal{L}},\mathcal{C})$ and $\epsilon$ is a closure operator  on $\mathcal{F}$, we see that $\sd(\Delta(\overline{\mathcal{L}},\mathcal{C}))$ is homotopy equivalent to $\Delta(\epsilon(\mathcal{F}))$. Let $\theta\colon \mathcal{F}\to \mathsf{Inf}(\mathcal{L})^*$ be the map taking the set $\{c_{i_1},\dots,c_{i_k}\}$ to the meet of its elements in $\mathcal{L}$. Observe that $\theta$ is a poset map and by restricting its domain we obtain a poset isomorphism between $\epsilon(\mathcal{F})$ and $\mathsf{Inf}(\mathcal{L})^*$. This gives another proof of Proposition~\ref{prop:inf} as the complexes $\sd(\Delta(\overline{\mathcal{L}},\mathcal{C}))$ and $\Delta(\mathsf{Inf}(\mathcal{L}))$ share the same homotopy type. Observe that the atoms of $\widehat{\mathcal{F}}$ correspond to the coatoms of $\mathcal{L}$ and for each $x\in \mathcal{F}$ the subposet $\widehat{\mathcal{F}}_{\leq x}$ is isomorphic to a Boolean algebra $B_{m}=2^{[m]}$, where $m$ is the length of a longest chain in $\widehat{\mathcal{F}}_{\leq x}$.
\end{rem}

Let $G$ be a group acting on the underlying set of a poset $\mathcal{P}$. We say $\mathcal{P}$ is a \emph{$G$-poset} if for all $g\in G$ and $S,T\in \mathcal{P}$ the relation $S \leq T$ implies $g\cdot S \leq g\cdot T$. In other words each $g\in G$ induces an automorphism of $\mathcal{P}$. 
Let $\mathcal{P}$ be a $G$-poset and $H$ be a subgroup of $G$. The \emph{$H$-fixed point subposet} $\mathcal{P}^H$ of $\mathcal{P}$ is defined as $$ \mathcal{P}^H := \{S\in \mathcal{P} \mid g\cdot S = S \text{ for all } g\in H\}. $$ 
And the \emph{$H$-singular subcomplex} $\Delta(\mathcal{P})^{H-\mathrm{sing}}$ of the order complex $\Delta(\mathcal{P})$ is the union $$ \Delta(\mathcal{P})^{H-\mathrm{sing}} := \bigcup_{J\in \mathcal{L}(H)_{>1}}\Delta(\mathcal{P}^J), $$ where $J$ is a nontrivial subgroup of $H$. Notice that $\Delta(\mathcal{P})^{H-\mathrm{sing}}$ may not be an induced subcomplex of $\Delta(\mathcal{P})$.

\begin{lem}[see {\cite[Lemma~4.4.14]{Smi11}}]\label{lem:fix}
  Let $G$ be a group, $H$ be a nontrivial subgroup of $G$, and $\mathcal{P}$ be a $G$-poset. Suppose $\Delta(\mathcal{P}^J)$ is contractible for any nontrivial subgroup $J$ of $H$. Then $\Delta(\mathcal{P})^{H-\mathrm{sing}}$ is contractible.
\end{lem}

An \emph{atom} of a finite lattice $\mathcal{L}$ is an element covering the unique minimum element and the lattice $\mathcal{L}$ is said to be \emph{atomic} if each element different from the unique minimum element is the join of some set of atoms of $\mathcal{L}$. Analogously, a \emph{coatom} is an element that is covered by the unique maximum of the lattice and the lattice is called \emph{coatomic} if any non-maximal element can be expressed as the meet of some set of coatoms.

\begin{lem}\label{lem:co}
  Let $\mathcal{L}$ be a finite lattice having at least two coatoms. Then $\widehat{\mathsf{Inf}(\mathcal{L})}$ is a coatomic lattice. If, further, $\mathcal{L}$ is atomic, then $\widehat{\mathsf{Inf}(\mathcal{L})}$ is atomic. Consequently, the lattice $\widehat{\mathsf{Inf}(\mathcal{L})}$ is both atomic and coatomic when $\mathcal{L}$ is the subrack lattice of a finite rack.
\end{lem}

\begin{proof}
  Let $\mathcal{C}=\{c_1,\dots,c_m\}$ be the set of coatoms of $\mathcal{L}$ and let $x=c_{i_1}\wedge\dots\wedge c_{i_k}$ and $y=c_{j_1}\wedge\dots\wedge c_{j_l}$ be elements in $\mathsf{Inf}(\mathcal{L})$. Let $z$ be the unique minimum $\hat{0}$ of $\widehat{\mathsf{Inf}(\mathcal{L})}$ if $c_{i_1}\wedge\dots\wedge c_{i_k}\wedge c_{j_1}\wedge\dots\wedge c_{j_l}$ is the unique minimum of $\mathcal{L}$. Otherwise, let $z$ be the element $c_{i_1}\wedge\dots\wedge c_{i_k}\wedge c_{j_1}\wedge\dots\wedge c_{j_l}$. Clearly, $z\in \mathsf{Inf}(\mathcal{L})$ is the meet of $x$ and $y$. More generally, the meet of any pair of elements is well-defined in $\mathsf{Inf}(\mathcal{L})$. Since $\mathcal{L}$ is a finite lattice, by \cite[Proposition~3.3.1]{Sta12}, $\mathsf{Inf}(\mathcal{L})$ is a lattice. By its definition the lattice $\mathsf{Inf}(\mathcal{L})$ is coatomic. This proves the first statement.

  Next, suppose $\mathcal{L}$ is atomic. Let $\mathcal{A}=\{a_1,\dots,a_n\}$ be the set of atoms of $\mathcal{L}$. For each $i$, $1\leq i\leq n$, let $b_i$ be the meet of the elements in $\{c\in \mathcal{C} \mid a_i\leq c\}$. Clearly, the elements $b_i$, $1\leq i\leq n$, form the set of minimal elements of $\mathsf{Inf}(\mathcal{L})$. Notice that $\hat{1} = b_{1}\vee\dots\vee b_{n}$. Let $x$ be an arbitrary element of $\mathsf{Inf}(\mathcal{L})$. Since $\mathcal{L}$ is atomic, $x$ is the join of some atoms of $\mathcal{L}$, say $x=a_{i_1}\vee\dots\vee a_{i_k}$. Since $x=b_{i_1}\vee\dots\vee b_{i_k}$, we deduce that $\widehat{\mathsf{Inf}(\mathcal{L})}$ is atomic.

  By \cite[Corollary~2.6]{KS19}, the subrack lattice of a finite rack is atomic. Hence, by the previous observations the last statement follows.
\end{proof}

\begin{rem}
  Although the meet operation on $\mathcal{L}$ coincides with the meet operation on $\widehat{\mathsf{Inf}(\mathcal{L})}$, the join operations on those two lattices may not coincide.
\end{rem}

Consider a finite $d$-dimensional simplicial comlex $\Delta$. We say $\Delta$ is \emph{pure} if any of its simplices contained in a face of dimension $d$, i.e., maximal faces (facets) have the same dimension. Also, we say a finite poset $\mathcal{P}$ is \emph{pure} (or \emph{graded}) of \emph{rank} $n$ if all the maximal chains in $\mathcal{P}$ have the same finite length $n=\ell(\mathcal{P})$. In such a case there exists a unique \emph{rank function} $\rho\colon \mathcal{P}\to \{0,1,\dots,n\}$ such that $\rho(x)=0$ whenever $x$ is minimal in $\mathcal{P}$ and $\rho(y) = \rho(x) + 1$ whenever $y$ covers $x$. We say the rank of $x\in\mathcal{P}$ is $\rho(x)$.

A finite simplicial complex $\Delta$ is \emph{shellable} if its maximal faces can be arranged in linear order $F_1,F_2,\dots,F_t$ in such a way that the subcomplex $\left(\bigcup_{i=1}^{k-1}F_i\right) \cap F_k$ is pure and ($\dim F_k - 1$)-dimensional for all $k = 2,3,\dots,t$. Shellability is defined originally for pure simplicial complexes. The general version we give here is usually called `nonpure' shellability in literature (see \cite{BW96,BW97}). A shellable simplicial complex may be homeomorphic to a nonshellable simplicial complex; therefore, being shellable is not  topologically invariant. However, shellability entails strong topological properties.

\begin{thm}[see {\cite[Theorem~4.1]{BW96}} or {\cite[Theorem~3.1.3]{Wac07}}]\label{thm:shell}
  A shellable simplicial complex has the homotopy type of a wedge of spheres (in varying dimensions), where for each $i$, the number of $i$-spheres is the number of $i$-facets whose entire boundary is contained in the union of the earlier facets. 
\end{thm}

\begin{rem}
  If a $d$-dimensional simplicial complex is pure and shellable, then it has the homotopy type of a wedge of $d$-spheres and it has nontrivial homology only in top dimension $d$.
\end{rem}

\section{Conjugacy class racks}\label{sec:ccr}

Consider the following question: For which racks $X$, the order complex $\Delta(\overline{\mathcal{R}(X)})$ has the homotopy type of a sphere? Theorem~\ref{thm:remA} below gives a partial answer to this question. Let $G$ be a group acting on a set $X$. We denote by $\mathsf{Orb}_G(a)$ the orbit of an element $a\in X$ under the action of $G$. If $S$ is a subset of $X$, we write $\mathsf{Stab}_G(S)$ to indicate the stabilizer of $S$.

\begin{thm}[see {\cite[Theorem~1]{Kay22}}]\label{thm:remA}
  Let $X$ be a finite rack having the following property: The poset $\overline{\mathcal{R}(X)}$ is non-empty and the equality $\mathsf{Stab}_{\mathsf{Inn}(X)}(M) = \mathsf{Inn}(X)$ holds for every maximal subrack $M$ of $X$. Then $\Delta(\overline{\mathcal{R}(X)})$ is homotopy equivalent to a $(m-2)$-sphere, where $m$ is the number of $\mathsf{Inn}(X)$-orbits in $X$.
\end{thm}

\begin{proof}
  Suppose $X$ is a finite rack having the stated property. Observe that a subrack $S$ of $X$ is maximal if and only if it is the union of all but one $\mathsf{Inn}(X)$-orbits in $X$. Therefore, there are exactly $m$ maximal subracks and the poset $\mathsf{Inf}(\mathcal{R}(X))$ is isomorphic to the proper part of the Boolean algebra $B_m=2^{[m]}$. By Proposition~\ref{prop:inf}, $\Delta(\overline{\mathcal{R}(X)})$ is homotopy equivalent to $\Delta(\mathsf{Inf}(\mathcal{R}(X)))$. Since $\Delta(\overline{B_m})$ is a triangulation of $S^{m-2}$, we see that $\Delta(\overline{\mathcal{R}(X)})$ has the homotopy type of a $(m-2)$-sphere.
\end{proof}

Let $G$ be a finite group. In \cite[Proposition~1.3]{HSW19} it is shown that $\Delta(\overline{\mathcal{R}(G)})$ has the homotopy type of a $(c-2)$-sphere, where $c$ is the number of conjugacy classes in $G$. The crucial observation made in there is that the maximal subracks of a group rack are unions of all but one conjugacy classes, and hence, the map $f$ taking a subrack $S$ of $G$ to the union of conjugacy classes that intersects $S$ in a nonempty set is a closure operator on $\overline{\mathcal{R}(G)}$. The original proof of Theorem~\ref{thm:remA} given in \cite{Kay22} imitates the same approach by observing that the map $f$ taking a subrack $S$ to the union of $\mathsf{Inn}(X)$-orbits having a representative element in $S$ is a closure operator  on $\overline{\mathcal{R}(X)}$ whenever $X$ has the property stated in Theorem~\ref{thm:remA}. Of course, $f(\overline{\mathcal{R}(X)})$ is the same as $\mathsf{Inf}(\mathcal{R}(X))$ in such a case.

\begin{rem}
  In \cite[Theorem~1.2]{HSW19} it is shown that the subrack lattice $\mathcal{R}(G)$ is pure if and only if $G$ is abelian or it is isomorphic to one of $S_3$, $D_8$, or $Q_8$. On the other hand, $\mathsf{Inf}(\mathcal{R}(G))$ is necessarily pure and shellable for any finite group $G$.
\end{rem}

Let $C$ be a non-central conjugacy class in a finite nilpotent group $G$. In \cite[Corollary~8]{Kay22} it is observed that $C$ satisfies the conditions of Theorem~\ref{thm:remA}; thus, $\mathsf{Inf}(\mathcal{R}(C))$ is isomorphic to $\overline{B_m}$, where $m$ is the number of maximal elements in $\mathcal{R}(C)$. Following result can be compared with \cite[Theorem~1.2]{HSW19}.

\begin{prop}\label{prop:nil}
  Let $G$ be a finite nilpotent group and $p$ be a prime number. Suppose the $p$-part of $|G|$ is $p^t$ and $t>2$. If $C$ is a conjugacy class in $G$ with $p^{t-2}$ elements, then the subrack lattice $\mathcal{R}(C)$ is pure and its length is $(p-1)t+1$. Moreover, for $1\leq k\leq t$, the set of elements $x\in\mathcal{R}(C)$ satisfying $$(p-1)(k-1)+1 \leq \rho(x)\leq (p-1)k$$ form a subposet of $\mathcal{R}(C)$ which is isomorphic to the disjoint union of $p^{t-k}$ many $\overline{B_p}$.  
\end{prop}

\begin{proof}
  Let $P$ be a Sylow $p$-subgroup of $G$. Since $G$ is nilpotent, $P$ is the unique Sylow $p$-subgroup of $G$ and, by assumption, its order is $p^t$, $t>2$. Moreover, $C$ is a conjugacy class of $P$ (see, for example, \cite[Chapter~5]{Rot95} for the basic properties of nilpotent groups). By \cite[Lemma~9]{Kay22}, $H=\langle C\rangle$ is a proper subgroup of $P$ and the action of $H$ on $C$ by conjugation is not transitive. Moreover, the maximal subracks of $C$ are unions of all but one $H$-orbits.

  Observe that the subgroup $H$ is normal in $P$ as it is generated by a conjugacy class of $P$. Therefore, the $H$-orbits form a system of blocks for $P$ (see \cite[Theorem~1.6A]{DM96}). In particular, the size of any $H$-orbit must be the same. Notice that the order of $H$ must be $p^{t-1}$ as it is generated by $C$ and $C$ contains $p^{t-2}$ elements by assumption.

  Take an element $a\in C$. The centralizer $C_P(a)$ is of order $p^2$ and since $a$ is not in the center of $P$, the orders of $a$ and $Z(P)$ are the same and equal to $p$. Since $H$ is a nilpotent group, the same argument applies, and so, the order of $C_H(a)$ is $p^2$ as well. However, that means the size of an $H$-orbit must be $p^{t-3}$. 

  Let $C'$ be an $H$-orbit and $K$ be the subgroup of $H$ generated by the elements of $C'$. Notice that the order of $K$ must be $p^{t-2}$. Consider the action of $K$ on $C'$. Similar arguments applied to $K$-orbits yield that the sizes of all $K$-orbits are the same. Let $C''$ be the $K$-orbit containing $a$. The size of $C''$ is $p^{t-4}$ and in total there are $p$ distinct $K$-orbits. Consider the action of $H$ on $K$-orbits. Since there are $p$ different $K$-orbits, we have $$ p=\frac{|H|}{|\mathsf{Stab}_H(C'')|}. $$ However, the last equality implies that $\mathsf{Stab}_H(C'')=K$ as the order of both groups is $p^{t-2}$. Let $D'$ be an $H$-orbit different from $C'$ and $L$ be the subgroup of $H$ generated by the elements of $D'$. Let $b$ be an element in $D'$. Suppose $b$ lies in $K$. Since $K$ and $L$ are both normal subgroups of $H$, their intersection is also normal in $H$. However, this implies that $K$ contains $C'\cup D'$. Let $U$ be the intersection of $K$ with $C$. From the previous statement $U$ is a union of at least two $H$-orbits. Clearly, $U$ cannot contain all of $C$ as in this case order of $K$ must be greater that $|C|=p^{t-2}$. Then $p^{t-3}<|U|<p^{t-2}$. Observe that conjugating $K$ with elements of $P$ we can cover $C$ as the disjoint union of conjugates of $U$. However, this is also not possible as $p$ is a prime number and there are in total $p$ distinct $H$-orbits.

  Let $B:=\langle b\rangle$ and consider $\mathsf{Orb}_B(a)$. Recall that $a\in C''$ and $b\in D'$. The arguments presented so far has the following implications: The element $b$ has order $p$ and it does not fix $C''$. Moreover, since the number of $K$-orbits in $C'$ is $p$, there are exactly $p$ elements in $\mathsf{Orb}_B(a)$ each from a different $K$-orbit. Notice that $a\in C''$ and $b\in C\setminus C'$ are arbitrarily chosen; thus, similar conclusions can be made for any element $a\in C''$ and $b\in C\setminus C'$.

  Same arguments apply for the other $H$-orbits (conjugates of $C'$). Also, similar deductions can be made at a deeper level by considering the subgroup generated by the elements in $C''$ and the partition of $C''$ into conjugacy classes of that subgroup. Accordingly, we may partition $C$ to $H$-orbits, then each $H$-orbit can be partitioned into smaller sets, etc. And those multilayered partitions determine the subracks generated by specific subsets of $C$. Accordingly, all the maximal chains in $\mathcal{R}(C)$ have the same length $(p-1)t+1$ and the subposet of $\mathcal{R}(C)$ that consists of all points whose ranks change from $(p-1)(k-1)+1$ to $(p-1)k$ is isomorphic to the disjoint union of $p^{t-k}$ many $\overline{B_p}$. 
\end{proof}

\begin{rem}
  Maybe the last statement in the above proof requires some explanation. For example, suppose $p=3$ and $t=3$. We can label the elements of $C$ with the integers $1$ to $27$ so that the multilayered partitions can be described succinctly as
  \begin{align*} &[\left[\left[[1][2][3]\right]\left[[4][5][6]\right]\left[[7][8][9]\right]\right]\\ &\left[\left[[10][11][12]\right]\left[[13][14][15]\right]\left[[16][17][18]\right]\right]\\
    &\left[\left[[19][20][21]\right]\left[[22][23][24]\right]\left[[25][26][27]\right]\right]].
  \end{align*}
  From this, for example, we conclude that the subrack generated by $1, 10$, and $19$ is the whole of $C$. Observe that every maximal chain in $\mathcal{R}(C)$ has the same length $7$. Figure~\ref{fig:nilcc} depicts the Hasse diagram of $\mathcal{R}(C)$.
\end{rem}

\begin{figure}[!ht]
  \centering
    \includegraphics[width=0.95\textwidth]{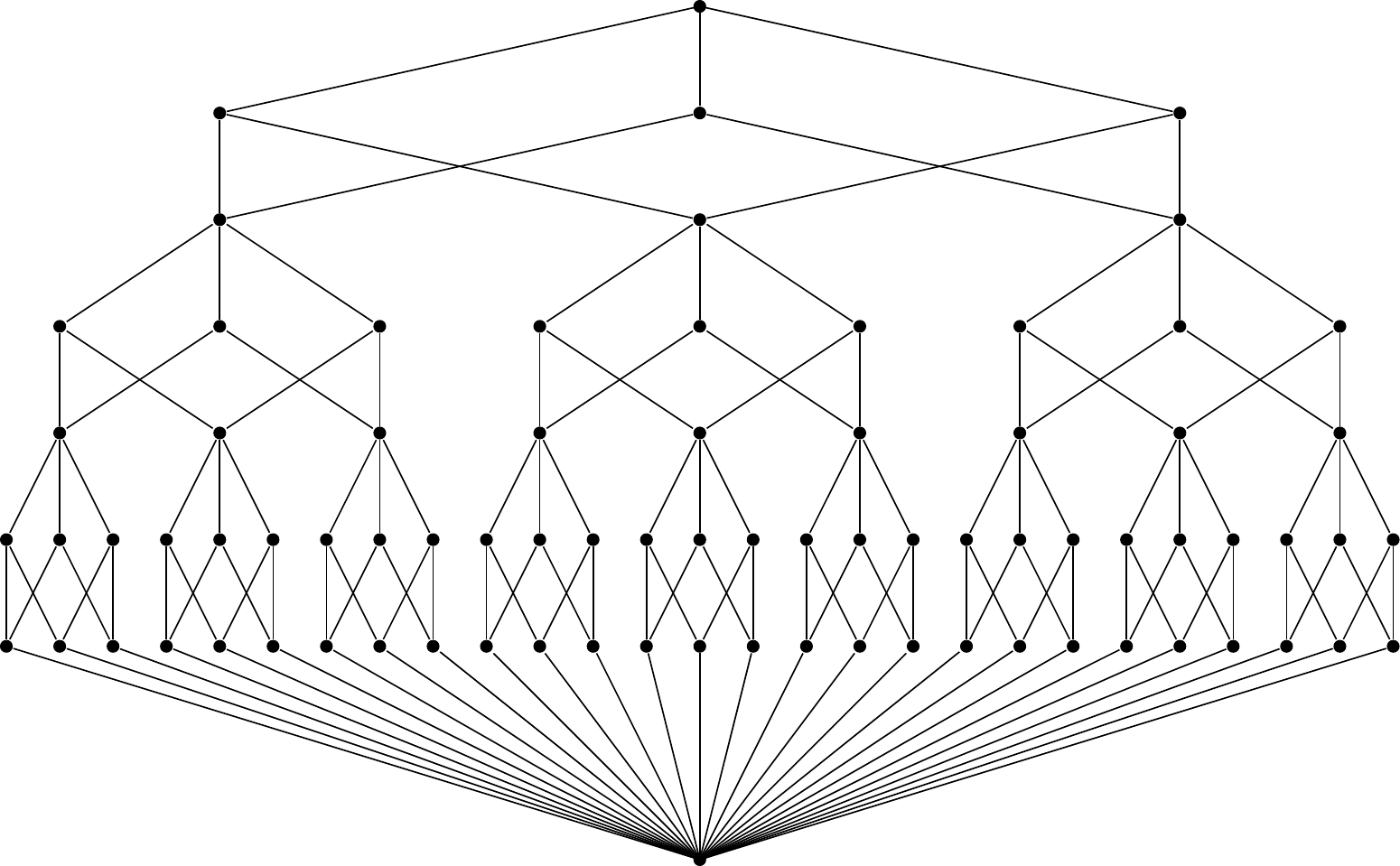}
    \caption{The subrack lattice of the conjugacy class $C$ with $27$ elements of a group of order $243$.}
    \label{fig:nilcc}
  \end{figure}

As is already observed any rack satisfying the property stated in Theorem~\ref{thm:remA} is necessarily not connected. Let $X$ be a rack and $S$ be a subrack of $X$. We define the \emph{orbit structure} of $S$ as the partition of $X$ into $\langle \Phi(S) \rangle$-orbits. Consider the map $\eta$ taking a subrack $S$ into $\widetilde{S}$, where $\widetilde{S}$ is the largest subrack of $X$ having the same orbit structure with $S$. Another observation made in \cite[Lemma~9]{Kay22} is that the map $\eta$ is a closure operator on $\overline{\mathcal{R}(X)}$ if $X$ is connected. Notice that the map $\pi\colon\eta(\mathcal{R}(X))\to \Pi_{|X|}$ taking subracks to their orbit structures (after relabeling elements of $X$ with $\{1,2,\dots,|X|\}$) is an injective poset map.

\begin{thm}\label{thm:pi}
  Let $X$ be a finite connected rack and $\mathcal{M}$ be the set of maximal subracks of $X$. Then $\mathsf{Inf}(\mathcal{R}(X))$ is a subposet of $\eta(\mathcal{R}(X))$; hence, $\mathsf{Inf}(\mathcal{R}(X))$ and $\pi(\mathsf{Inf}(\mathcal{R}(X)))$ are isomorphic posets. Moreover, $\mathsf{Inf}(\mathcal{R}(X))$ is isomorphic to a subposet of $\mathsf{Inf}(\Pi_{|X|},\pi(\mathcal{M}))$.
\end{thm}

\begin{proof}
  Let $\mathcal{N}$ be a set of maximal subracks of $\mathcal{R}(X)$ and $S$ be the intersection of the subracks in $\mathcal{N}$. Suppose $S\neq \emptyset$. Notice that $S$ is an arbitrary element of $\mathsf{Inf}(\mathcal{R}(X))$. We want to show that $S=\widetilde{S}$. If $\mathcal{N}$ is a singleton, the equality $S=\widetilde{S}$ must hold as $X$ is a connected rack. Suppose there are more than one elements in $\mathcal{N}$. Let $a$ be an element of $X$ preserving the orbit structure of $S$, i.e., $a\triangleright O = O$ for any orbit $O$ in the orbit structure of $S$. Clearly, $a$ preserves orbit structures of each element in $\mathcal{N}$; hence, is an element of $S$ as it lies in each of the maximal subracks containing $S$. In other words $S=\widetilde{S}$. Since $\eta$ is a closure operator on $\overline{\mathcal{R}(X)}$ and $S\in\eta(\overline{\mathcal{R}(X)})$, the restriction of $\eta$ to the subposet $\mathsf{Inf}(\mathcal{R}(X))$ is just the identity function. Moreover, since $\pi\colon\eta(\mathcal{R}(X))\to \Pi_{|X|}$ is an injective poset map, we see that there is a bijective correspondence between $\mathsf{Inf}(\mathcal{R}(X))$ and the image set $\pi(\mathsf{Inf}(\mathcal{R}(X)))$. This proves the first statement.

  Recall that $\mathcal{M}$ is defined as the set of maximal subracks of $X$. Next, we shall show that $\mathsf{Inf}(\mathcal{R}(X))$ is isomorphic to a subposet of $\mathsf{Inf}(\Pi_{|X|},\pi(\mathcal{M}))$. To this end let's define the map $f\colon \mathsf{Inf}(\mathcal{R}(X)) \to \mathsf{Inf}(\Pi_{|X|},\pi(\mathcal{M}))$ in the following way: For an element $S\in \mathsf{Inf}(\mathcal{R}(X))$, the map $f$ takes $S$ to the the meet of the orbit structures of the maximal subracks of $X$ containing $S$. One can easily observe that $f$ is a poset map. Let $T$ be an element of $\mathsf{Inf}(\mathcal{R}(X))$ that is different from $S$. Without loss of generality we may assume that $T$ is contained by a maximal subrack of $X$ which does not contain $S$. However, this implies $f(S)$ and $f(T)$ are different partitions as $S$ would not preserve the partition $f(T)$ in such a case.

\end{proof}

\begin{example}\label{ex:ex1}
  Let $A=\{a,b,c,d\}$ and $S_A$ be the symmetric group on $A$. By \cite[Example~2.3]{HSW19} we know that $\mathcal{R}(T)$ is isomorphic to the partition lattice $\Pi_4$, where $T$ is the conjugacy class of the transposition $(a,b)$ in $S_A$. Notice that $\overline{\mathcal{R}(T)}$ and $\mathsf{Inf}(\mathcal{R}(T))$ are the same posets. There are six elements of $T$ and we relabel them as $$ (a,b)\to 1,\, (c,d)\to 2,\, (a,c)\to 3,\, (b,d)\to 4,\, (a,d)\to 5,\, (b,c)\to 6. $$ Let $\mathcal{M}$ be the set of maximal subracks of $T$. By Theorem~\ref{thm:pi} the proper part of the subrack $\mathcal{R}(T)$ is isomorphic to the poset $\pi(\mathsf{Inf}(\mathcal{R}(T)))$ (see Figure~\ref{fig:ex1}).

\begin{figure}[!ht]
  \centering
    \includegraphics[width=0.95\textwidth]{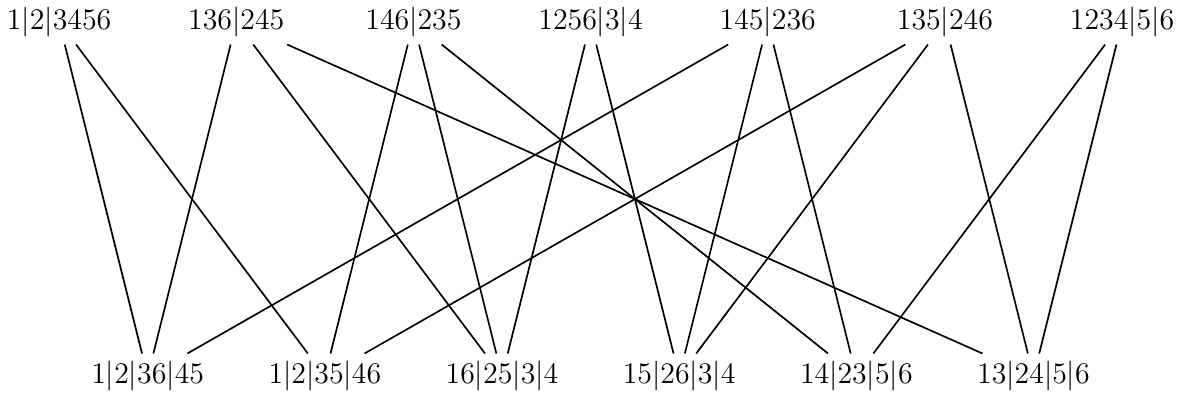}
    \caption{The poset $\pi(\mathsf{Inf}(\mathcal{R}(T)))$ in Example~\ref{ex:ex1}.}
    \label{fig:ex1}
  \end{figure}

\end{example}

It is natural to ask for a connected rack $X$ whether $\mathsf{Inf}(\mathcal{R}(X))$ and $\eta(\overline{\mathcal{R}(X)})$ are the same subposets of $\mathcal{R}(X)$. The next example shows that this is not necessarily the case.

\begin{example}\label{ex:ex2}
  The group $$ G = \langle a,b,c \mid a^4 = b^4 = c^3 = (ac)^3 = 1, ab=ba, cac^{-1}=b \rangle $$ is of order $48$ and the conjugacy class $C$ of the element $c$ consists of $16$ elements, namely $$ C = \{c,ac,a^2c,a^3c,bc,abc,a^2bc,a^3bc,b^2c,ab^2c,a^2b^2c,a^3b^2c,b^3c,ab^3c,a^2b^3c,a^3b^3c\}. $$ All connected quandles having less than $48$ elements are available in \textsf{Rig}~\cite{rig}, a \textsf{GAP} package designed for computations related to racks and quandles. The conjugacy class $C$ is connected as a quandle and it is isomorphic to SmallQuandle(16,1) in \textsf{Rig}. 

  \begin{figure}[!ht]
  \centering
  \includegraphics[width=0.95\textwidth]{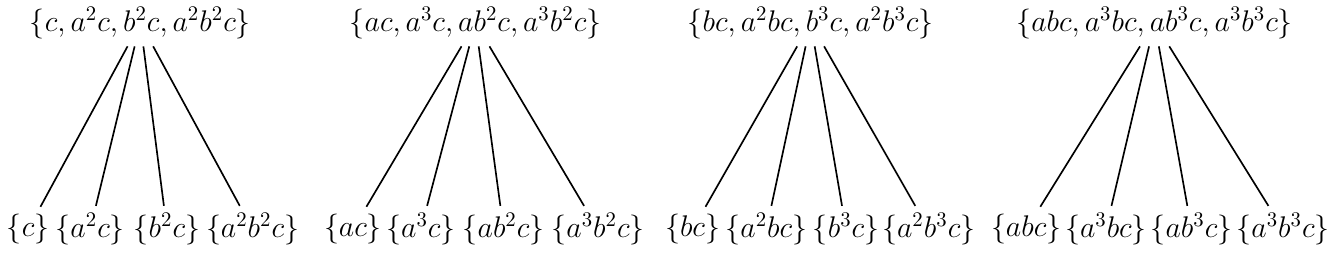}
    \caption{The proper part of the subrack lattice of the conjugacy class $C$ in Example~\ref{ex:ex2}.}
    \label{fig:ex2}
  \end{figure}

  As is apparent from Figure~\ref{fig:ex2}, there is exactly one proper subrack of $C$ containing $c$ properly. Moreover, the orbit structures of $\{c\}$ and $\{c,a^2c,b^2c,a^2b^2c\}$ are different.
  
\end{example}

For a finite connected rack $X$, it is possible to determine whether the subrack lattice $\mathcal{R}(X)$ is pure by focusing on a smaller part of $\mathcal{R}(X)$.

\begin{lem}\label{lem:filter}
  Let $\mathcal{P}$ be a finite poset and $x$ be a minimal element of $\mathcal{P}$. Suppose the automorphism group of $\mathcal{P}$ acts on the set of minimal elements of $\mathcal{P}$ transitively. Then $\mathcal{P}$ is pure if and only if $\mathcal{P}_{\geq x}$ is pure.
\end{lem}

\begin{proof}
  Any interval in $\mathcal{P}$ is isomorphic to an interval in $\mathcal{P}_{\geq x}$ and vice versa.
\end{proof}

\begin{cor}
  Let $X$ be a finite connected rack, $S$ be a minimal element of $\overline{\mathcal{R}(X)}$. Then $\mathcal{R}(X)$ is pure if and only if $\overline{\mathcal{R}(X)}_{\geq S}$ is pure.
\end{cor}

\begin{proof}
  Since $\mathcal{R}(X)$ is pure if and only if $\overline{\mathcal{R}(X)}$ is pure, it is enough to show that $\overline{\mathcal{R}(X)}$ is pure if and only if $\overline{\mathcal{R}(X)}_{\geq S}$ is pure.
  
  For an element $a\in X$, let $[a]$ be the set $\{\phi_a^n(a)\mid n\in \mathbb{Z}\}$. Those sets form a partition of $X$. Let $Q:=\{[a]\mid a\in X\}$ and for $[a],[b]\in Q$ define $[a]*[b]:=[a\triangleright b]$. As is shown in \cite{KS19}, the binary operation $*$ is well-defined, the set $Q$ together with $*$ is a quandle, and the subrack lattice of $X$ is isomorphic to the subrack lattice of $Q$. Further, one can easily observe that $Q$ is connected if $X$ is connected. Therefore, it is enough to prove the statement for quandles.

  Let $X$ be a finite connected quandle. By the third rack axiom (A3), the minimal elements of $\overline{\mathcal{R}(X)}$ are exactly the $1$-element subsets of $X$. Since $X$ is connected, the inner automorphism group $\mathsf{Inn}(X)$ acts on the set of singletons in $\mathcal{R}(X)$ transitively. Hence, Lemma~\ref{lem:filter} applies.
\end{proof}

There are many examples of connected racks $X$ such that the poset $\mathsf{Inf}(\mathcal{R}(X))$ is pure even when $\mathcal{R}(X)$ is not pure. The following example shows that a reverse situation is also possible.

\begin{example}
  Let $C$ be the conjugacy class of $(1,2,3)(4,5,6)$ in the alternating group $A_6$ of degree $6$. Let $$ a = (1,2,3)(4,5,6),\; b = (1,2,3)(4,6,5),\; c = (1,3,2)(4,5,6),\; d = (1,3,2)(4,6,5).  $$
  Clearly, $\langle a,b \rangle$ is a Sylow $3$-subgroup of $A_6$. Since any two distinct Sylow $3$-subgroups generate the whole of $A_6$, the chain $$ \{a\} < \{a,b\} < \{a,b,c\} < \{a,b,c,d\} $$ is maximal in $\overline{\mathcal{R}(C)}$.

  Let $e = (1,2,6)(3,4,5)$ and $f = (1,2,4)(3,6,5)$. Then $H = \langle a,e \rangle$ is isomorphic to $A_4$ and $K = \langle a,f\rangle$ is isomorphic to $A_5$. One way to observe this is to consider the action of $S_6$ on the six Sylow $5$-subgroups and construct an (injective) homomorphism from $S_5$ to $S_6$ and look for the inverse images of $a,e$, and $f$ under this homomorphism. (Such an homomorphism can be used to construct the outer automorphisms of $S_6$.) The set $H\cap C$ splits into two conjugacy classes, say $C_1$ and $C_2$, in $H$ and $$ \{a\} < C_1 < C_2 < K\cap C $$ is a maximal chain in $\overline{\mathcal{R}(C)}$. Further computations show that $\mathcal{R}(C)$ is a pure lattice. 

 Next, we consider $\mathsf{Inf}(\mathcal{R}(C))$. Any maximal subrack of $C$ containing $a$ also contains $a^2=d$ and $$ \{a,d\} < \{a,b,c,d\} $$ is a maximal chain in $\mathsf{Inf}(\mathcal{R}(C))$. Another maximal chain is this one: $$ \{a,d\} < H\cap C < K\cap C. $$ Since the length of those two maximal chains is different, $\mathsf{Inf}(\mathcal{R}(C))$ is not pure.

\end{example}

\begin{question}
  Is there any finite rack $X$ such that the poset $\mathsf{Inf}(\mathcal{R}(X))$ is not pure and the homotopy type of $\Delta(\mathsf{Inf}(\mathcal{R}(X)))$ is a wedge of top dimensional spheres?
\end{question}

Let $G$ be a group acting on the underlying set of a poset $\mathcal{P}$. We say $\mathcal{P}$ is a \emph{$G$-poset} if for all $g\in G$ and $S,T\in \mathcal{P}$ the relation $S \leq T$ implies $g\cdot S \leq g\cdot T$. In other words each $g\in G$ induces an automorphism of $\mathcal{P}$. Let $\mathcal{P}$ be a $G$-poset. We define the poset $\mathsf{Orb}_G(\mathcal{P})$ in the following way: The elements of $\mathsf{Orb}_G(\mathcal{P})$ are the set of orbits of elements of $\mathcal{P}$ under the action of $G$ and $\mathsf{Orb}_G(S) \leq \mathsf{Orb}_G(T)$ whenever there exist $U\in \mathsf{Orb}_G(S)$ and $V\in \mathsf{Orb}_G(T)$ such that $U < V$. It can be easily observed that the map $\omega\colon \mathcal{P}\to \mathsf{Orb}_G(\mathcal{P})$ taking elements of the poset $\mathcal{P}$ to their orbits is a poset map. 

The \emph{integer partition lattice} $I_n$ is defined in the following way: A typical element of $I_n$ is a multiset of integers whose total sum is $n$ and $\{i_1,\dots,i_r\} \leq \{j_1,\dots,j_s\}$ if there is a partition of $\{i_1,\dots,i_r\}$ into $s$ blocks $B_1|\dots|B_s$ and each $j_t$, $1\leq t\leq s$, is the sum of the elements in the block $B_t$. Let $\mathcal{P}$ be a subposet of the partition lattice $\Pi_n$. The map $\iota\colon \mathcal{P}\to I_n$ taking a partition of $[n]$ to the integer partition that consists of the sizes of the blocks in the partition is clearly a poset map.

\begin{thm}\label{thm:ipi}
  Let $X$ be a finite connected rack with $n$ elements, $G:=\mathsf{Inn}(X)$ and $\mathcal{I} := \pi(\mathsf{Inf}(\mathcal{R}(X)))$. For an orbit structure $P\in \mathcal{I}$, let $[P]$ be the set of all elements in $\mathcal{I}$ having the same multiset of block sizes with $P$. Then
  \begin{enumerate}[(i)]
  \item the maps $\omega\circ\pi\colon \mathsf{Inf}(\mathcal{R}(X))\to \mathsf{Orb}_G(\mathcal{I})$ and $\iota\circ\pi\colon \mathsf{Inf}(\mathcal{R}(X))\to I_n$ are poset maps, and
  \item the poset $\mathsf{Orb}_G(\mathcal{I})$ is isomorphic to the poset $\iota(\mathcal{I})$ if and only if $G$ acts on each $[P]$, $P\in \mathcal{I}$, transitively. 
  \end{enumerate}
\end{thm}

\begin{proof}
  \emph{Statement (i):} Each of the maps in the compositions $\omega\circ\pi$ and $\iota\circ\pi$ are poset maps as previously observed.

  \emph{Statement (ii):} The class $[P]$ of partitions having the same set of block sizes with $P$ is exactly the orbit of $P$ under the action of $G:=\mathsf{Inn}(X)$ when $G$ acts transitively on $[P]$. And each partition class $[P]$ can be uniquely identified with $\iota(P)$.
\end{proof}

\section{$p$-Power racks}\label{sec:ppr}

Let $G$ be a finite group and let $X\subseteq G$ be a conjugation rack. The set $Z:=X\cap Z(G)$, i.e., the set $Z$ of all central elements of $G$ that are lying in $X$, is clearly a subrack of $X$. With this notation we have the following lemma.

\begin{lem}[see {\cite[Lemma~2.9]{HSW19}}]\label{lem:chi}
  Suppose both $Z:=X\cap Z(G)$ and $X\setminus Z$ are non-empty sets. Then the subrack lattice $\mathcal{R}(X)$ is isomorphic to the direct product $\mathcal{R}(X\setminus Z) \times \mathcal{R}(Z)$. Consequently, the equality $$ \widetilde{\chi}(\Delta(\overline{\mathcal{R}(X)})) = (-1)^{|Z|}\cdot \widetilde{\chi}(\Delta(\overline{\mathcal{R}(X\setminus Z)})) $$ holds, where $\widetilde{\chi}(\cdot)$ denotes the reduced Euler characteristic. 
\end{lem}

\begin{proof}
  The map $f\colon \mathcal{R}(X)\to \mathcal{R}(Y) \times \mathcal{R}(Z)$ taking $S\in\mathcal{R}(X)$ to $(S\cap Y, S\cap Z)$, where $Y:=X\setminus Z$, is a poset isomorphism. Since $\mathcal{R}(Z)$ is isomorphic to the Boolean algebra $B_{|Z|}$, we have $\widetilde{\chi}(\Delta(\overline{\mathcal{R}(Z)})) = (-1)^{|Z|}$. Then, the latter statement follows from \cite[Propositions~3.8.2 and 3.8.6]{Sta12}.
\end{proof}

For a finite group $G$ and a prime number $p$, we define the $p$-power rack $G_p$ as the conjugation rack of all elements of $G$ whose order is a power of $p$. Clearly, the underlying set of $G_p$ is the union of Sylow $p$-subgroups of $G$.

\begin{example}\label{ex:g3a4}
  Let $G:=A_4$ be the alternating group on $4$ letters and consider $G_3$, the rack of order three elements of $G$ together with the identity element. Clearly $|G_3|=9$ and there are four Sylow $3$-subgroups $P_i,\,1\leq i\leq 4,$ isomorphic to the cyclic group with three elements. Let
\begin{align*}
  &P_1=\{(123),(132),1\}, \\
  &P_2=\{(134),(143),1\}, \\
  &P_3=\{(142),(124),1\}, \\
  &P_4=\{(243),(234),1\}.  
\end{align*}
Then the following two conjugacy classes together with $1$ form $G_3$:
\begin{align*}
  &C=\{(123),(134),(142),(243)\} \\
  &D=\{(132),(143),(124),(234)\} 
\end{align*}
Notice that any two non-identity elements from distinct Sylow $3$-subgroups generate $G$, as $\langle P_i,P_j\rangle=G$ if $i\neq j$. In Figure~\ref{fig:g3a4} we depict the Hasse diagram of $\overline{\mathcal{R}(G_3)}$. 
\end{example}

\begin{sidewaysfigure}
\centering
\includegraphics[width=\textwidth]{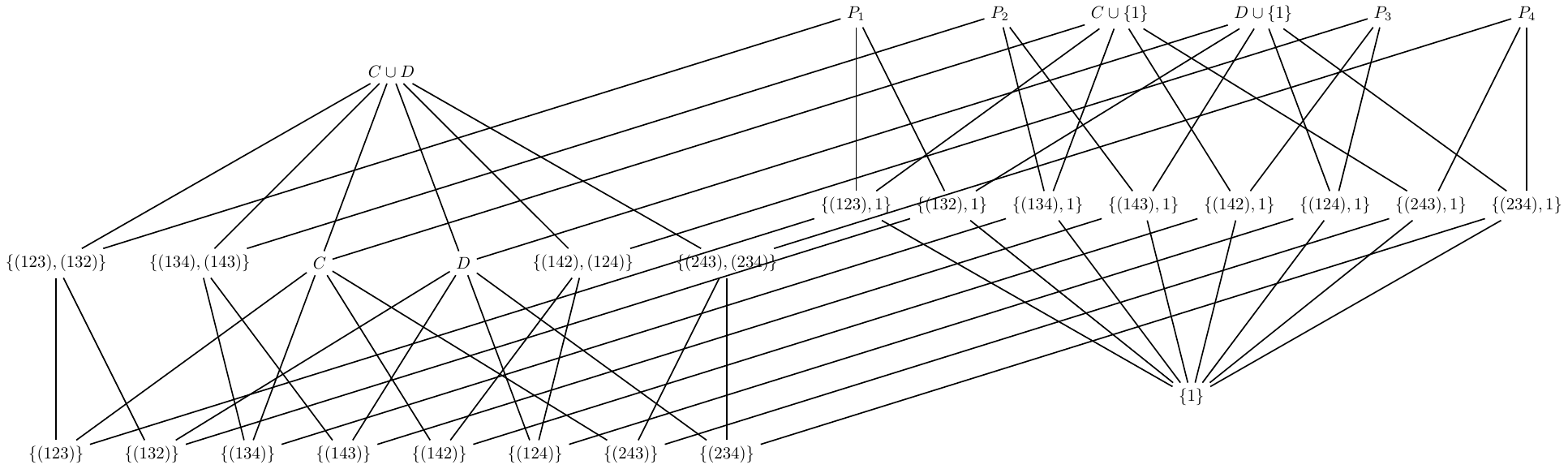}
\caption{The poset of proper non-empty subracks of $3$-power elements of $A_4$}
\label{fig:g3a4}
\end{sidewaysfigure}

\begin{rem}
  The subrack lattice of the conjugation rack $G_3$ in the previous example is observed to be pure. Let $H$ be the alternating group $A_5$. As is observed in \cite[Example~2.6]{HSW19} the subrack lattice of $H_5$ is not pure. From Figure~\ref{fig:g3a4} it is obvious that $\mathcal{R}(G_3)$ is coatomic and, hence, $\overline{\mathcal{R}(G_3)}$ and $\mathsf{Inf}(\mathcal{R}(G_3))$ are the same. On the other hand, $\overline{\mathcal{R}(H_5)}$ and $\mathsf{Inf}(\mathcal{R}(H_5))$ are not the same and, in particular, $\mathsf{Inf}(\mathcal{R}(H_5))$ is pure. Let $K$ be the alternating group $A_6$ and consider $\mathsf{Inf}(\mathcal{R}(K_2))$. Let $C$ be the conjugacy class of $(1,2,3,4)(5,6)$ in $K$ together with the identity element $1$ and $D$ be the conjugacy class of $(1,2)(3,4)$ together with $1$. Let $P = \langle (1,2,3,4)(5,6), (1,2)(3,4)\rangle$. Clearly, $P$ is a Sylow $2$-subgroup of $K$ and it is isomorphic to the dihedral group $D_8$. Since Sylow $2$-subgroups are maximal subracks of $K_2$ and since there is no Sylow $2$-subgroup of $K$ containing $\langle (1,2,3,4)(5,6)\rangle$ other than $P$, we see that
  $$ \{1\} < P\cap C < P $$
  is a maximal chain in $\mathsf{Inf}(\mathcal{R}(K_2))$. Let $ Q = \langle (1,3,2,4)(5,6), (1,3)(2,4)\rangle$ and $R = \langle (1,5,3,6)(2,4), (1,5)(3,5)\rangle$. Then
  $$ \{1\} < P\cap Q\cap R < P\cap Q < P\cap C < P $$
is another maximal chain in $\mathsf{Inf}(\mathcal{R}(K_2))$; hence, $\mathsf{Inf}(\mathcal{R}(K_2))$ is not pure. 
\end{rem}

In Lemma~\ref{lem:chi} we derived an arithmetic condition relating the Euler characteristic of $\Delta(\overline{\mathcal{R}(X)})$ with the characteristics of other order complexes. The following result is substantially stronger as the homotopy type of $\Delta(\overline{\mathcal{R}(X)})$ is expressed in terms of the homotopy types of other complexes.

\begin{prop}\label{prop:dir}
  Let $G$ be a finite group and let $X\subseteq G$ be a conjugation rack. Suppose both $Z:=X\cap Z(G)$ and $X\setminus Z$ are non-empty sets. Then
  $$ \Delta(\overline{\mathcal{R}(X)}) \simeq \Delta(\mathsf{Inf}(\mathcal{R}(X\setminus Z))) * S^{|Z|-1}.  $$
\end{prop}

\begin{proof}
  We already observed in Lemma~\ref{lem:chi} that the lattice $\mathcal{R}(X)$ is isomorphic to the direct product $\mathcal{R}(X\setminus Z) \times \mathcal{R}(Z)$. By \cite[Proposition~4.3]{BW83}, $\Delta(\overline{\mathcal{R}(X)})$ is homotopy equivalent to the suspension of the join $\Delta(\overline{\mathcal{R}(X\setminus Z)}) * \Delta(\overline{\mathcal{R}(Z)})$. By Proposition~\ref{prop:inf} the complexes $\Delta(\overline{\mathcal{R}(X)})$ and $\Delta(\mathsf{Inf}(\mathcal{R}(X)))$ share the same homotopy type. Since $\Delta(\mathcal{R}(Z))$ is a Boolean algebra $B_{|Z|}$, the result follows.
\end{proof}

We shall introduce the following two subposets of the subrack lattice of a conjugation rack $X\subseteq G$.
\begin{itemize}
\item $\mathcal{S}(X) := \{S\in\mathsf{Inf}(\mathcal{R}(X)) \mid S \text{ is a union of some } \mathsf{Inn}(X)\text{-orbits}\}$
\item $\mathcal{P}(X) := \{P\in\mathsf{Inf}(\mathcal{R}(X)) \mid P\cap S\neq \emptyset  \text{ for all } S\in \mathcal{S}(X)\}$
\end{itemize}
Let $\mathcal{S}$ be the set of all coatoms of $\mathcal{R}(X)$ which are union of all but one $\mathsf{Inn}(X)$-orbits. It is easy to observe that $\mathcal{S}(X)$ and $\mathsf{Inf}(\mathcal{R}(X),\mathcal{S})$ are the same posets. The subracks in $\mathcal{S}(X)$ are called \emph{spherical}. Notice that $\mathcal{S}(X)$ is isomorphic to the proper part of a Boolean algebra and, hence, $\Delta(\mathcal{S}(X))$ is homeomorphic to a sphere. The subracks in $\mathcal{P}(X)$ are said to be \emph{parabolic}. Two particular cases worth mentioning:
\begin{enumerate}[(i)]
\item If $X=G$ is a group rack, then there are no parabolic subracks, i.e., $\mathcal{P}(X)$ is the empty poset, as any subset of a group that contains a representative element from each conjugacy class necessarily generates the whole group.
\item If $X=C$ is a connected conjugacy class rack, then any element of $\mathsf{Inf}(\mathcal{R}(X))$ would be vacuously parabolic as $\mathcal{S}(X)$ would be the empty poset in that case.
\end{enumerate}

\begin{rem}
  The underlying sets of $\mathcal{S}(X)$ and $\mathcal{P}(X)$ are disjoint for any conjugation rack $X$. Moreover, parabolic subracks cannot contain a conjugacy class $C$ if $\mathsf{Inn}(X)$ is a non-abelian simple group and $C$ is different from $\{1\}$. To see this observe that an $\mathsf{Inn}(X)$-orbit of an element of $X$ is a conjugacy class of $\mathsf{Inn}(X)$; hence, generates a normal subgroup.
\end{rem}

Given two disjoint posets $\mathcal{P}$ and $\mathcal{Q}$, their \emph{ordinal sum} $\mathcal{P}\oplus \mathcal{Q}$ is the poset whose underlying set is $\mathcal{P}\sqcup\mathcal{Q}$ and $x\leq y$ in $\mathcal{P}\oplus \mathcal{Q}$ if and only if one of the following holds:
\begin{enumerate}[(i)]
\item $x\leq y$ in $\mathcal{P}$
\item $x\leq y$ in $\mathcal{Q}$
\item $x\in\mathcal{P}$ and $y\in\mathcal{Q}$
\end{enumerate}

Observe that simplices of $\Delta(\mathcal{P}\oplus \mathcal{Q})$ are exactly the simplices of $\Delta(\mathcal{P}) * \Delta(\mathcal{Q})$. As a consequence, we have the following Lemma.

\begin{lem}\label{lem:ord}
  For any two disjoint posets $\mathcal{P}$ and $\mathcal{Q}$, we have $$ \Delta(\mathcal{P}\oplus \mathcal{Q}) = \Delta(\mathcal{P}) * \Delta(\mathcal{Q}). $$
\end{lem}

The next Proposition justifies the usefulness of the notions of spherical and parabolic subracks.

\begin{prop}\label{prop:ord}
  Let $G$ be a finite group and let $X\subseteq G$ be a conjugation rack. Then $$ \Delta(\overline{\mathcal{R}(X)}) \simeq \Delta(\mathcal{S}(X)\oplus \mathcal{P}(X)). $$
\end{prop}

\begin{proof}
  By Proposition~\ref{prop:inf} we know that $\Delta(\overline{\mathcal{R}(X)})$ is homotopy equivalent to $\Delta(\mathsf{Inf}(\mathcal{R}(X)))$. So it is enough to show that $\Delta(\mathsf{Inf}(\mathcal{R}(X)))$ and $\Delta(\mathcal{S}(X)\oplus \mathcal{P}(X))$ share the same homotopy type. Let $f\colon \mathsf{Inf}(\mathcal{R}(X)) \to \mathcal{S}(X)\oplus \mathcal{P}(X)$ be the map defined in the following way: If $A\in \mathsf{Inf}(\mathcal{R}(X))$ does not intersect at least one $\mathsf{Inn}(X)$-orbit in $X$, then $f(A)$ is the union of the $\mathsf{Inn}(X)$-orbits of the elements in $A$. If $A\in \mathsf{Inf}(\mathcal{R}(X))$ is a parabolic subrack, then $f(A)=A$. Clearly, $f$ is a poset map.
  Let $B\in \mathcal{S}(X)\oplus \mathcal{P}(X)$. Next, we consider the fiber $f^{-1}((\mathcal{S}(X)\oplus \mathcal{P}(X))_{\leq B})$. If $B$ is spherical, then the fiber $f^{-1}((\mathcal{S}(X)\oplus \mathcal{P}(X))_{\leq B})$ would be a cone with apex $B$, hence, is contractible. If $B$ is parabolic, then the fiber would be conically contractible. To see this observe that the map taking an element $C\in f^{-1}((\mathcal{S}(X)\oplus \mathcal{P}(X))_{\leq B})$ to $C\cap B$ defines a closure operator from the fiber to itself whose image is a cone. The result follows from Quillen's fiber theorem (see \cite[Theorem~1.6]{Qui78}).
\end{proof}

Combining the previous results gives the following Theorem on $p$-power racks.

\begin{thm}\label{thm:p}
  Let $G$ be a finite group and $p$ be a prime number dividing the order of $G$. Let $c$ be the number of conjugacy classes of $p$-power elements in $\langle G_p \rangle$. Then $$ \Delta(\overline{\mathcal{R}(G_p)}) \simeq S^{c-2} * \Delta(\mathcal{P}(G_p)). $$ 
\end{thm}

\begin{proof}
  Clearly, the complex $\Delta(\mathcal{S}(G_p))$ is homeomorphic to a $(c-2)$-sphere. So, the proof follows from Proposition~\ref{prop:ord} and Lemma~\ref{lem:ord}.
\end{proof}

\begin{rem}
  Let $\mathcal{Q}$ be the poset which is obtained by adding a unique minimal element $\hat{0}$ to $\mathcal{P}(G_p)$. Observe that $\Delta(\mathcal{Q})$ is contractible. Let $f\colon \mathcal{S}(G_p)\oplus \mathcal{P}(G_p) \to \mathcal{Q}$ be the poset map which takes any element of $\mathcal{S}(G_p)$ to $\hat{0}$ and any element of $\mathcal{P}(G_p)$ to itself. Then $f$ satisfies the conditions of Theorem~\ref{thm:fiber}, i.e., for any $B\in\mathcal{Q}$ the fiber $\Delta(f^{-1}(\mathcal{Q}_{\leq B}))$ is $\ell(f^{-1}(\mathcal{Q}_{<B}))$-connected. Clearly, the fiber is contractible whenever $B\neq \hat{0}$. And if $B=\hat{0}$, we have $\Delta(f^{-1}(\mathcal{Q}_{\leq B}))\simeq S^{c-2}$ and $\Delta(\mathcal{Q}_{>B})\simeq \Delta(\mathcal{P}(G_p))$. This gives an alternative proof of Theorem~\ref{thm:p}.
\end{rem}

\begin{rem}
  By Lemma~\ref{lem:co} we know that the bounded poset $\widehat{\mathsf{Inf}(\mathcal{R}(X))}$ is a coatomic lattice. However, $\widehat{\mathcal{P}(X)}$ may not even be a lattice as the intersection of some maximal elements of $\mathcal{R}(X)$ that are containing a representative element from each $\mathsf{Inn}(X)$-orbit may be a nonempty subrack of $X$ that does not lie in $\mathcal{P}(X)$.
\end{rem}

\begin{cor}[see {\cite[Proposition~1.3]{HSW19} and Theorem~\ref{thm:remA}}]
  Let $G$ be a finite group and $P$ be a Sylow $p$-subgroup of $G$. Suppose $P$ is a nontrivial normal subgroup of $G$. Then $\Delta(\overline{\mathcal{R}(G_p)})$ is homotopy equivalent to a $(c-2)$-sphere, where $c$ is the number of conjugacy classes in $P$.
\end{cor}

\begin{proof}
  Since $P=G_p$ is a group rack, $\mathcal{P}(G_p)$ would be the empty poset. By Theorem~\ref{thm:p} the complex $\Delta(\overline{\mathcal{R}(G_p)})$ is homotopy equivalent to a $(c-2)$-sphere.
\end{proof}

\begin{rem}
  Notice that $\Delta(\overline{\mathcal{R}(G_p)})$ would be homotopy equivalent to a sphere whenever $G$ is nilpotent and $p$ divides the order of $G$.
\end{rem}

\begin{cor}\label{cor:p}
  Let $G$ be a finite group and $p$ be a prime number dividing the order of $G$. Let $c$ be the number of conjugacy classes of $p$-power elements in $\langle G_p \rangle$ and $k>1$ be the number of Sylow $p$-subgroups. Suppose Sylow $p$-subgroups are the only elements of $\mathcal{P}(G_p)$. Then $\Delta(\overline{\mathcal{R}(G_p)})$ is homotopy equivalent to the wedge of $k-1$ many spheres of dimension $c-1$.
\end{cor}

\begin{proof}
  Clearly, $\Delta(\mathcal{P}(G_p))$ is a wedge of $k-1$ many $0$-spheres. The result follows from Theorem~\ref{thm:p}.
\end{proof}

\begin{rem}
  In Example~\ref{ex:g3a4} the $3$-power rack $G_3$ of the alternating group $G=A_4$ is studied (see Figure~\ref{fig:g3a4}). In $G_3$ there are $3$ conjugacy classes and $\Delta(\mathcal{P}(G_3))$ consists of four isolated points. By Corollary~\ref{cor:p} the homotopy type of $\Delta(\overline{\mathcal{R}(G_3)})$ is the wedge of $3$ spheres of dimension $2$.
\end{rem}

\begin{example}
  Let $G$ be the dihedral group of order $30$. As is well known this group has a  presentation $G = \langle a, b \mid a^{15} = b^2 = 1, bab = a^{14} \rangle$ and the number of conjugates of $b$ is $15$. Hence $G_2 = C\cup \{1\}$ where $C$ is the conjugacy class of $b$. One can easily observe that the minimal elements of $\mathcal{P}(G_2)$ are exactly the Sylow $2$-subgroups of $G$. Moreover, any proper subgroup of $G$ containing a Sylow $2$-subgroup of $G$ as a proper subgroup corresponds to a maximal element of $\mathcal{P}(G_2)$ and vice versa. Figure~\ref{fig:d30} depicts the Hasse diagram of $\mathcal{P}(G_2)$. In this figure the 15 nodes lying on the lower plane represents the minimal elements. There are on total $30$ edges in this figure and $22$ of those edges are already labeled with the integer from $1$ to $22$. If we further label the remaining $8$ edges in any order with the integers from $23$ to $30$, this would yield a shelling order for the facets of $\Delta(\mathcal{P}(G_2))$. Moreover, in this shelling the boundary of a facet is contained in the union of earlier facets if and only if the facet is among the last eight. By Theorem~\ref{thm:shell}, the homotopy type of $\Delta(\mathcal{P}(G_2))$ is the wedge of $8$ circles. Then, by Theorem~\ref{thm:p}, the homotopy type of $\Delta(\overline{\mathcal{R}(G_2)})$ is the wedge of $8$ spheres each of dimension $2$.

    \begin{figure}[!ht]
      \centering
      \includegraphics[width=0.95\textwidth]{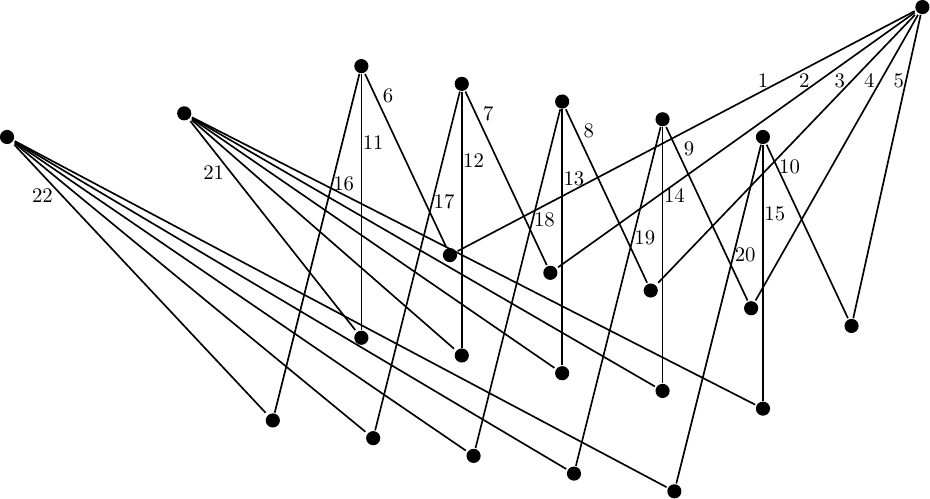}
      \caption{The poset of parabolic elements in $\mathsf{Inf}(\mathcal{R}(G_2))$, where $G$ is the dihedral group of order $30$.}
      \label{fig:d30}
    \end{figure}

\end{example}

\begin{example}
  Let $G=A_5$ be the alternating group and consider the $2$-power rack $G_2$. In $A_5$ there are $15$ elements of order $2$ and those elements form a single conjugacy class, say $C$. Also Sylow $2$-subgroups are elementary abelian in this group. Hence, similar to the previous example, $G_2$ is the union of $C$ and $\{1\}$. Notice that the minimal elements of $\mathcal{P}(G_2)$ are exactly the subgroups of $A_5$ of order $2$. Let us relabel the elements of $C$ in the following way:
    \begin{align*}
      & (1,2)(3,4)\to a,\, (1,2)(3,5)\to b,\, (1,2)(4,5)\to c,\, (1,3)(2,4)\to d,\, (1,3)(2,5)\to e, \\
      & (1,3)(4,5)\to f,\, (1,4)(2,3)\to g,\, (1,4)(2,5)\to h,\, (1,4)(3,5)\to i,\, (1,5)(2,3)\to j, \\
      & (1,5)(2,4)\to k,\, (1,5)(3,4)\to l,\, (2,3)(4,5)\to m,\, (2,4)(3,5)\to n,\, (2,5)(3,4)\to o.
    \end{align*}
    Figure~\ref{fig:g2a5} represents the Hasse diagram of $\mathcal{P}(G_2)$ (for the sake of succinctness we did not specify the identity element $1$ when we describe the corresponding subrack in the labels of the nodes in this figure). Unlike the previous example $\mathcal{P}(G_2)$ is not pure. However, $\mathsf{Inf}(\mathcal{P}(G_2))$ is pure and of dimension $1$. The boldface nodes in the figure correspond to the elements of $\mathsf{Inf}(\mathcal{P}(G_2))$. Further computations show that the homotopy type of $\Delta(\mathcal{P}(G_2))$ is the wedge of $40$ circles and, hence,  the homotopy type of $\Delta(\overline{\mathcal{R}(G_2)})$ would be the wedge of $40$ spheres each of dimension $2$.

    \begin{sidewaysfigure}
      \centering
      \includegraphics[width=\textwidth]{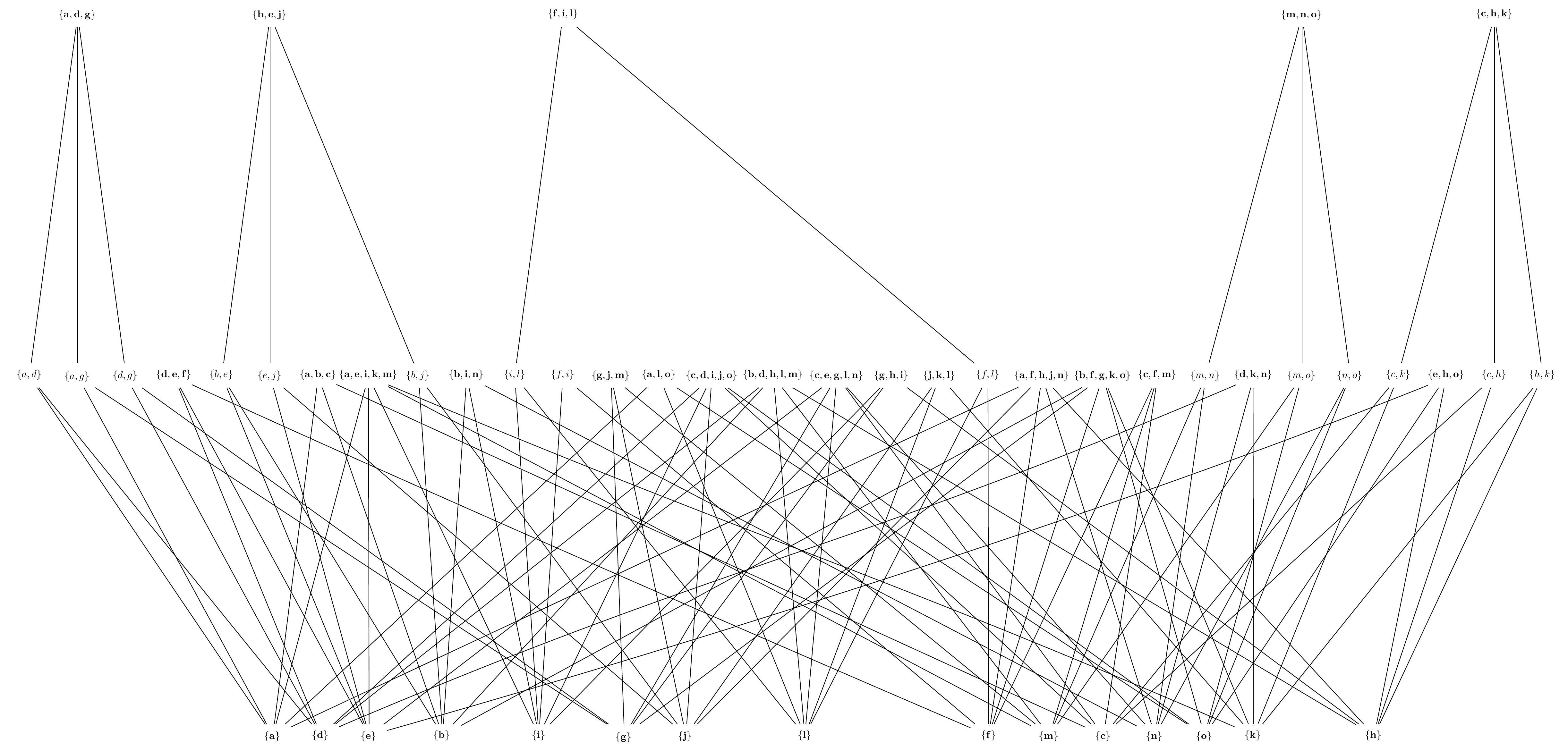}
      \caption{The poset of parabolic elements in $\mathsf{Inf}(\mathcal{R}(G_2))$, where $G$ is the alternating group $A_5$.}
      \label{fig:g2a5}
    \end{sidewaysfigure}

\end{example}

In the last example we observe that there are maximal elements of $\mathcal{P}(G_2)$ which are not union of some Sylow $2$-subgroups when $G=A_5$. However, any parabolic element in $\mathsf{Inf}(\mathcal{R}(G_p))$ must necessarily be a union of Sylow $p$-subgroups if the minimal elements of $\mathcal{P}(G_p)$ are exactly the Sylow $p$-subgroups of $G$. Hence, we may ask the following question.

\begin{question}
  For which finite groups $G$ and prime numbers $p$, the minimal elements of $\mathcal{P}(G_p)$ are exactly the Sylow $p$-subgroups of $G$?
\end{question}

Let $G$ be a group and $p$ be a prime number dividing $|G|$. Recall that we define $\mathcal{S}_p(G)$ as the poset of nontrivial $p$-subgroups of $G$. An important result in the field of subgroup complexes, nicknamed homological Sylow theorem, states that the Euler characteristic $\chi(\Delta(\mathcal{S}_p(G)))\equiv 1$ modulo $|G|_p$, where $|G|_p$ is the order of a Sylow $p$-subgroup of $G$ (see \cite{Br75}). Equivalently, $|G|_p$ divides the reduced Euler characteristic $\tilde{\chi}(\Delta(\mathcal{S}_p(G)))$. Since $\mathcal{R}(G_p)$ is just the rack analogue of the poset $\mathcal{S}_p(G)$, it is natural to ask whether a similar result is valid in the context of subrack lattices. The next result tells us that this is indeed the case when $\mathcal{P}(G_p)$ has a certain property. The proof is also similar to the proof of the homological Sylow theorem (see \cite[Section~4]{Qui78} and \cite[Proof~1 of Theorem~5.3.1]{Smi11}). Recall that the \emph{$p$-core} $O_p(G)$ of $G$ is defined as the largest normal $p$-subgroup of $G$. Obviously $O_p(G)$ coincides with the $p$-power rack $G_p$ exactly when $G$ has a normal Sylow $p$-subgroup.

\begin{thm}\label{thm:euler}
  Let $G$ be a finite group and $p$ be a prime number dividing $|G|$. Suppose for any  $p$-subgroup $J$ of $G$ the set $\{S\in \mathcal{P}(G_p)\mid J\subseteq S\}$ has a unique minimum element. Then $$ \tilde{\chi}(\Delta(\overline{\mathcal{R}(G_p)})) \equiv 0 \mod{|G|_p}\; \text{ if and only if }\; O_p(G) \neq G_p. $$
\end{thm}

\begin{proof}
  Suppose $P$ is the normal Sylow $p$-subgroup of $G$. In such a case $\tilde{\chi}(\Delta(\overline{\mathcal{R}(G_p)}))$ would be either $+1$ or $-1$ as the subrack complex of a group rack is homotopy equivalent to a sphere by \cite[Proposition~1.3]{HSW19}.

  Next, suppose $G$ has a non-normal Sylow $p$-subgroup $P$. In such a case $\Delta(\mathcal{P}(G_p))$ would be non-empty. Let $p^k$ be the order of $P$. Our aim is to show that $p^k$ divides $\tilde{\chi}(\Delta(\overline{\mathcal{R}(G_p)}))$.

  For a simplicial complex $\Delta$ we denote by $\#\Delta^d$ the number of simplices of $\Delta$ of dimension $d$. Let $m$ be the dimension of $\Delta(\mathcal{S}(G_p))$ and $n$ be the dimension of $\Delta(\mathcal{P}(G_p))$. Consider the product
  \begin{align*}
    (-1)\tilde{\chi}(\Delta(\mathcal{S}(G_p)))\cdot \tilde{\chi}(\Delta(\mathcal{P}(G_p))) & = (-1) \sum\limits_{i=-1}^m(-1)^i\#\Delta(\mathcal{S}(G_p))^i\sum\limits_{j=-1}^n(-1)^j\#\Delta(\mathcal{P}(G_p))^j \\
    & = \sum\limits_{d=-2}^{m+n}(-1)^{d+1}\sum\limits_{i+j=d}\#\Delta(\mathcal{S}(G_p))^i\#\Delta(\mathcal{P}(G_p))^j
  \end{align*}
  Observe that $$ \sum\limits_{i+j=d}\#\Delta(\mathcal{S}(G_p))^i\#\Delta(\mathcal{P}(G_p))^j = \#\Delta(\mathcal{S}(G_p)\oplus \mathcal{P}(G_p))^{d+1}. $$
  Since $\tilde{\chi}(\Delta(\mathcal{S}(G_p)))$ is either $+1$ or $-1$, by Proposition~\ref{prop:ord} it is enough to show that $p^k$ divides $\tilde{\chi}(\Delta(\mathcal{P}(G_p)))$.

  Let $P$ be a Sylow $p$-subgroup of $G$. We shall show that the $P$-singular subcomplex $\Delta(\mathcal{P}(G_p))^{P-\mathrm{sing}}$ of the complex of parabolic subracks is contractible. Let $J$ be a nontrivial subgroup of $P$ and $f\colon \mathcal{P}(G_p)^J\to \mathcal{P}(G_p)^J$ be the map taking $S\in \mathcal{P}(G_p)^J$ to the meet of all coatoms of $\mathcal{R}(G_p)$ containing $J\vee S$ in $\mathcal{R}(G_p)$. Notice that the normalizer $N_G(S)$ of $S$ contains both $J$ and $S$ and its intersection with $G_p$ must be properly contained in $G_p$ as $S$ is a parabolic subrack. This shows that the map $f$ is well-defined and, further, it is a closure operator on $\mathcal{P}(G_p)^J$. Obviously $f(\mathcal{P}(G_p)^J)$ is a cone; hence, $\Delta(\mathcal{P}(G_p)^J)$ is contractible. Then, by Lemma~\ref{lem:fix}, $\Delta(\mathcal{P}(G_p))^{P-\mathrm{sing}}$ is contractible. In particular $\tilde{\chi}(\Delta(\mathcal{P}(G_p))^{P-\mathrm{sing}}) = 0$.

  Take a parabolic subrack $T$ which is not a vertex in $\Delta(\mathcal{P}(G_p))^{P-\mathrm{sing}}$. Since the stabilizer of $T$ under the action of $P$ is trivial, $P$ acts on the set of simplices containing $T$ freely. Therefore, we have the congruence $$ \#(\Delta(\mathcal{P}(G_p))^{P-\mathrm{sing}})^i\equiv \#\Delta(\mathcal{P}(G_p))^i \mod{p^k} $$
which implies the desired result.  
\end{proof}

\begin{rem}
  Unfortunately, we couldn't prove Theorem~\ref{thm:euler} in full generality. Such a general Theorem would imply that \emph{the complex $\Delta(\overline{\mathcal{R}(G_p)})$ is homotopy equivalent to a sphere if and only if $O_p(G) = G_p$}. This last statement was pointed out by Volkmar Welker in a private communication and it can be considered as the rack analogue of a conjecture by Quillen in \cite{Qui78} stating $\Delta(\mathcal{S}_p(G))$ is contractible if and only if $O_p(G)\neq 1$.
\end{rem}


\end{document}